\documentclass[12pt]{amsart}
\usepackage{amsfonts,amssymb,amscd,xy,bm}
\usepackage[leqno]{amsmath}
\usepackage{mathrsfs}
\usepackage{amssymb}
\usepackage{amsmath}
\usepackage{amsxtra}
\usepackage{amsthm}
\usepackage[backref]{hyperref} 
\numberwithin{equation}{section}

\usepackage[a4paper, total={5.5in, 7in}]{geometry}

\def\pic{{\rm{Pic}}\,}

\newcommand{\isoto}{\overset{\!\sim}{\to}}

\newcommand{\et}{{\rm {\acute et}}}

\DeclareMathAlphabet{\mathbbmsl}{U}{bbm}{m}{sl}

\newcommand{\abe}{{\rm{\bf Ab}}}
\newcommand{\fl}{{\rm fl}}
\newcommand{\spp}{S^{\lle\le\prime}}
\newcommand{\sppp}{S^{\le\lle\prime\prime}}

\def\sfs{S_{\fl}^{\e\sim}}

\def\bg{{\mathbb G}}

\def\si{\Sigma}

\def\hom{{\rm{Hom}}\e}

\usepackage[usernames,dvipsnames]{xcolor}
\usepackage{tabularx}
\usepackage{epsfig,amssymb}
\usepackage{bm} 
\usepackage{verbatim} 

\usepackage{hyperref}
\definecolor{labelkey}{rgb}{1,0,0}

\newcommand{\spec}{\mathrm{ Spec}\,}

\newcommand{\s}{\mathscr }

\def\ofs{\mathcal O_{\lbe\lbe F\lbe,\e \si}}

\def\oks{\mathcal O_{\lbe\lbe K,\e \Sigma}}

\def\okfs{\mathcal O_{\lbe\lbe K,\e \Sigma}\lbe\otimes_{\e\mathcal O_{\lbe\lbe F\be,\e \si}}\!\mathcal O_{\lbe K,\e \Sigma}}

\def\e{\kern 0.06em}
\def\be{\kern -.1em}
\def\le{\kern 0.03em}
\def\lle{\kern 0.015em}
\def\lbe{\kern -.025em}

\newcommand{\sh}{\kern -.4em\phantom{a}^{\mathbf{\sim}}}

\newcommand{\lra}{\longrightarrow}

\def\be{\kern -.1em}
\def\le{\kern 0.03em}
\def\lle{\kern 0.04em}
\def\lbe{\kern -.025em}

\newcommand{\krn}{\mathrm{Ker}\e}
\newcommand{\cok}{\mathrm{Coker}\e}

\def\e{\kern 0.08em}

\newtheorem{lemma}{Lemma}[section]
\newtheorem{theorem}[lemma]{Theorem}
\newtheorem{proposition-definition}[lemma]{Proposition-Definition}
\newtheorem{corollary}[lemma]{Corollary}
\newtheorem{proposition}[lemma]{Proposition}
\theoremstyle{definition}
\newtheorem{definition}[lemma]{Definition}

\theoremstyle{remark}
\newtheorem{remark}[lemma]{Remark}

\newtheorem{examples}[lemma]{Examples}

\begin{document}

\input xy     
\xyoption{all}

\title[\v{Cech} cohomology and capitulation]{\v{Cech} cohomology and the Capitulation kernel}

\subjclass[2010]{Primary 11R29; Secondary 14F20}
	
\author{Cristian D. Gonz\'alez-Avil\'es}
\address{Departamento de Matem\'aticas, Universidad de La Serena,
La Serena, Chile} \email{cgonzalez@userena.cl}

\keywords{Capitulation kernel, global units, \v{C}ech cohomology}
	
\thanks{The author was partially supported by Fondecyt grant
1160004}

\begin{abstract} We discuss the capitulation kernel associated to a degree $n$ covering using \v{C}ech cohomology and the Kummer sequence. The main result is a five-term exact sequence that relates the capitulation kernel to the \v{C}ech cohomology of the $n$-th roots of unity and a certain subquotient of the group of units modulo $n$-th powers.
\end{abstract}
	
\maketitle

\section{Introduction}

Let $F$ be a global field, let $\si$ be a nonempty finite set of primes of $F$ containing the archimedean primes and let $\ofs$ and $C_{F,\e\si}$ denote, respectively, the ring of $\si\le$-\le integers and the $\si$-ideal class group of $F$. Now let $K/F$ be a finite Galois extension with Galois group $\Delta$ of order $n\geq 2$ and let $\oks$ and $C_{K,\e\si}$ denote, respectively, the ring of $\si_{K}\le$-\le integers and the $\si_{K}$-ideal class group of $K$, where $\si_{K}$ denotes the set of primes of $K$ lying over the primes in $\si$. 
The extension of ideals from $F$ to $K$ defines a map
\begin{equation}\label{ccap}
j_{K\be/\lbe F,\e\si}\colon C_{F,\e\si}\to C_{K,\e\si}
\end{equation}
which is often called the {\it capitulation map}, after Arnold Scholz. An ideal class in $C_{F,\e\si}$ is said to {\it capitulate} in $K$ if it lies in $\krn\e j_{K\be/\lbe F,\e\si}$, which is often called the {\it capitulation kernel} associated to the pair $(K/F,\si\le)$. It was shown in \cite[Theorem 2.3]{ga07} that $\krn\e j_{K\be/\lbe F,\e\si}$ is naturally isomorphic to the kernel of the canonical localization map in $\Delta$-cohomology
\[
H^{\e 1}\lbe(\Delta,\oks^{\e *}\lle)\to \prod_{v\notin \si} H^{\e 1}\lbe(\Delta_{\e w_{v}},\mathcal O_{\!w_{v}}^{\e*}\lbe),
\]
where, for each prime $v\notin \si$, $w_{v}$ denotes a fixed prime of $K$ lying over $v$ and $\Delta_{\e w_{v}}$ is the corresponding decomposition subgroup of $\Delta$. Unfortunately, the $n$-torsion abelian group $H^{\e 1}\lbe(\Delta,\oks^{\e *}\lle)$ (and therefore also $\krn\e j_{K\be/\lbe F,\e\si}$) is difficult to compute and can vary widely. In this paper we obtain additional information on $\krn\e j_{K\be/\lbe F,\e\si}$ that should be helpful in determining its structure, at least in some particular cases of interest. To explain our results, let $S=\spec \ofs$, $\spp=\spec \oks$ and note that, via the well-known isomorphisms $C_{F,\e\si}\simeq \pic\e S\simeq H^{\e 1}\lbe(S_{\et},\bg_{m})$ (and similarly for $\spp\e$), the map $j_{K\be/\lbe F,\e\si}$ can be identified with the canonical restriction map in \'etale  cohomology $j^{(1)}\colon H^{\e 1}\lbe(S_{\et},\bg_{m})\to H^{\e 1}\lbe(\spp_{\et},\bg_{m})$. Now, it is well-known \cite[III, proof of Proposition 4.6, p.~123]{mi1} that $\krn\e j^{(1)}$ is canonically isomorphic to the first \v{C}ech cohomology group $\check{H}^{1}(\spp\be/\lbe S,\bg_{m})$ associated to the fppf covering $\spp\be/\lbe S$. Thus there exists a canonical isomorphism of finite $n$-torsion abelian groups $\krn\e j_{K\be/\lbe F,\e\si}\simeq \check{H}^{1}(\spp\be/\lbe S,\bg_{m})$. This observation motivated the  central idea of this paper, which is that the tools afforded by the \v{C}ech cohomology theory and the cohomology sequences arising from the Kummer sequence $0\to\mu_{\le n}\to\bg_{m}\overset{\!n}{\lra}\bg_{m}\to 0$ can be combined to yield new information on $\krn\e j_{K\be/\lbe F,\e\si}$.

Set
\begin{equation}\label{cool0}
\Psi(K/F,\Sigma\e)=\{\e u\in \oks^{\e *}\colon u\otimes u^{-1}\in ((\okfs)^{*})^{n}\e\}.
\end{equation}
Then the following holds.

\begin{theorem} {\rm ($=$Proposition \ref{mcor2})} There exists a canonical exact sequence of finite $n$-torsion abelian groups
\[
\begin{array}{rcl}
0&\to& \ofs^{\e *}\be\cap\be (\oks^{\e *}\e)^{n}\be/(\ofs^{\e *}\e)^{n}\to \check{H}^{\le 1}(\oks\e/\ofs, \mu_{\le n})\to \krn\e j_{K\be/\lbe F,\e\si}\\
&\to& \Psi(K/F,\Sigma\le)/\ofs^{\e *}\e (\oks^{\e *})^{n}\to \check{H}^{\le 2}(\oks\e/\ofs, \mu_{\le n}),
\end{array}
\]
where $\Psi(K/F,\Sigma\e)$ is the subgroup of $\oks^{\e *}$ given by \eqref{cool0}.
\end{theorem}

The following corollary of the proposition is part of Corollary \ref{mcor3}.

\begin{corollary}\label{mcor3.0} Assume that $\Sigma$ contains all primes of $F$ that ramify in $K$. Assume, furthermore, that $n$ is odd and $\mu_{\le n}(K\e)=\{1\}$. Then there exists  a canonical isomorphism of finite $n$-torsion abelian groups
\[
\krn\e j_{K\be/\lbe F,\e\si}\isoto \Psi(K/F,\Sigma\le)/\ofs^{\e *}\e (\oks^{\e *})^{n}.
\]
In particular, at most $[\e\oks^{\e *}\e\colon \ofs^{\e *}\e (\oks^{\e *})^{n}\e]$ $\Sigma$-ideal classes of $F$ capitulate in $K$. 
\end{corollary}

\smallskip

The preceding discussion, which is a summary of the contents of the last section of the paper (section \ref{last}), is a particular case of the more general developments of section \ref{scap}\,. Our main theorem, namely Theorem \ref{big}, applies to {\it admissible pairs} $(G,f\e)$ (see Definition \ref{adm}\,), of which $(\bg_{m,\e S}, \spp\be/S\e)$ is an example. See Examples \ref{exs} for additional examples of admissible pairs. Section \ref{pre} consists of various preliminaries.

\section{Preliminaries}\label{pre}

\subsection{Generalities} The category of abelian groups will be denoted by $\abe$. If  $n\geq 1$ is an integer and $A$ is an object of an abelian category $\s A$, $A_{\e n}$ (respectively, $A/ n\le$) will denote the kernel (respectively, cokernel) of the multiplication by $n$ morphism on $A$. If $\psi\colon A\to B$ is a morphism in $\s A$, $\psi_{\le n}\colon A_{\e n}\to B_{\le n}$ and $\psi/n\colon A/n\to B/n$ will denote the morphisms in $\s A$ induced by $\psi$. Note that
\begin{equation}\label{kn}
\krn (\psi_{\le n})=(\krn\psi)_{n}.
\end{equation}
However, $\krn (\psi/n)\neq (\krn\psi)/n$ (in general).

\smallskip

{\it All schemes appearing in this paper are tacitly assumed to be non-empty.}

\smallskip

If $S$ is a scheme and $\tau$ denotes either the \'etale ($\et$) or fppf ($\fl$) topology on $S$, $S_{\tau}$ will denote the small $\tau$ site over $S$. Thus $S_{\tau}$ is the category of $S$-schemes that are \'etale (respectively, flat) and locally of finite presentation over $S$ equipped with the \'etale (respectively, fppf) topology. If $f\colon \spp\to S$ is a faithfully flat morphism locally of finite presentation, then $f$ is an fppf covering of $S$. We will write $S_{\tau}^{\e\sim}$ for the (abelian) category of sheaves of abelian groups on $S_{\tau}$. If $G$ is a commutative $S$-group scheme, then the presheaf represented by $G$ is an object of $S_{\tau}^{\e\sim}$. In particular, if $f\colon \spp\to S$ is as above, the induced map $G(\lle f\lle)\colon G(S\e)\hookrightarrow G(\spp)$ is an injection that will be regarded as an inclusion. If $n\geq 1$, the object $G_{n}$ of $S_{\tau}^{\e\sim}$  is represented by the $S$-group scheme $G\times_{n_{\le G},\e G,\e\varepsilon}S$, where $n_{\le G}\colon G\to G$ is the $n$-th power morphism on $G$ and $\varepsilon\colon S\to G$ is the unit section of $G$.

Let $\varGamma_{\!S}\colon S_{\tau}^{\e\sim}\to\abe$ be the $S$-section functor on $S_{\tau}^{\e\sim}$. For every object $\s F$ of $S_{\tau}^{\e\sim}$ and every integer $r\geq 0$, let $H^{\le r}\lbe(S_{\tau},\s F\le)=R^{\e r}\be\varGamma_{\!S}(\s F)$ be the corresponding $r$-th $\tau$ cohomology group of $\s F$. If $G$ is a commutative $S$-group scheme (regarded as an object of $S_{\tau}^{\e\sim}$) and $\spp\to S$ is a morphism of schemes, we will write $H^{\le r}\lbe(\spp_{\tau},G\e)$ for $H^{\le r}\lbe(\spp_{\tau},G_{\be S^{\lle\prime}}\be)$, where $G_{S^{\lle\prime}}=G\times_{S}\spp$. The abelian group $H^{\le 1}(S_{\tau},G\e)$ will be identified with the group of isomorphism classes of right (sheaf) $G$-torsors over $S$ relative to the $\tau$ topology on $S$. If $G=\bg_{m,\e S}$ (respectively, $G=\mu_{\le n,\e S}$ for some integer $n\geq 2$), the groups $H^{\le r}\lbe(S_{\tau},G\e)$ will be denoted by $H^{\le r}\lbe(S_{\tau},\bg_{m})$ (respectively, $H^{\le r}\lbe(S_{\tau},\mu_{\e n})$). If $G$ is smooth over $S$, the groups $H^{\le r}\lbe(S_{\fl},G\e)$ and $H^{\le r}\lbe(S_{\et},G\e)$ will be identified via \cite[Theorem 11.7(1), p.~180]{dix}. 
If $f\colon\spp\to S$ is a morphism of schemes, the {\it $r$-th restriction morphism associated to $(G,f\e)$} is the morphism of abelian groups
\begin{equation}\label{ires}
j^{(r)}=R^{\e r}\be\varGamma_{\!S}(f\le)\colon H^{\le r}\lbe(S_{\tau},G\e)\to H^{\le r}\lbe(\spp_{\tau},G\e). 
\end{equation}
When reference to $G$ and $\tau$ (respectively, $G, \e S^{\lle\prime}\be/ S$ and $\tau$) is necessary, $j^{(r)}$ will be denoted by $j_{\le G,\e\tau}^{(r)}$ (respectively, $j_{\le G,\e S^{\lle\prime}\be/ S,\e\tau}^{(r)}$). If $G$ is smooth over $S$, the maps $j_{\le G,\e\fl}^{(r)}$ and $j_{\le G,\e\et}^{(r)}$ will be identified and denoted by  $j_{\le G}^{(r)}$, i.e.,
\begin{equation}\label{sm}
j_{\le G}^{(r)}=j_{\le G,\e\fl}^{(r)}=j_{\le G,\e\et}^{(r)}\qquad\text{(if $G$ is smooth over $S\e$)}.
\end{equation}
Note that $j^{(0)}$ is the inclusion $G(\lle f\lle)\colon G(S\e)\hookrightarrow G(\spp)$. The map
\begin{equation}\label{cap1}
j^{(1)}\colon H^{\e 1}\lbe(S_{\tau},G\e)\to H^{\e 1}\lbe(\spp_{\tau},G\e)
\end{equation}
sends the class in $H^{\e 1}\lbe(S_{\tau},G\e)$ of a right (sheaf) $G$-torsor $X$ over $S$ (relative to the $\tau$ topology on $S\e$) to the class in $H^{\e 1}\lbe(\spp_{\tau},G\e)$ of the right (sheaf) $G_{\be S^{\lle\prime}}\e$-\e torsor $X_{\lbe S^{\lle\prime}}$ over $\spp$. A class $[\lle X\lle]\in H^{\e 1}\lbe(S_{\tau},G\e)$ lies in $\krn\e j^{(1)}$ if, and only if, $X(\spp\e)\neq\emptyset$. See \cite[III, comments after Proposition 4.1, p.~120]{mi1}. After the classical case $G=\bg_{m,\e \mathcal O}$, where $\mathcal O$ is the ring of integers of a number field, we call $\krn j_{G,\e\et}^{(1)}$ the {\it capitulation kernel} associated to the pair $(G,f\e)$. 

\smallskip

For every scheme $X$, the {\it group of global units on $X$} is the abelian group
\begin{equation}\label{gu}
U(X)=\varGamma(X, \mathcal O_{\lbe X}\lbe)^{*}=\bg_{m,\le X}(X\lle).
\end{equation}
We will identify $\pic X$ and $H^{\le 1}(X_{\et},\bg_{m})$ via  \cite[Theorem 4.9, p.~124]{mi1}.
A morphism of schemes $f\colon \spp\to S$ induces a morphism of abelian groups
\begin{equation}\label{pmap}
\pic\be f\colon \pic S\to \pic \spp, [\e E\e]\mapsto [\e E\be\times_{S}\be \spp\,],
\end{equation}
where $[\e E\e]$ denotes the isomorphism class of the right (sheaf) $\bg_{m,\le S}$-torsor $E$ and $[\e E\be\times_{S}\be \spp\,]$ denotes the isomorphism class of the right (sheaf) $\bg_{m,\le S^{\prime}}\le$-\le torsor $E\be\times_{S}\be \spp$.

\subsection{\v{C}ech cohomology} If $n\geq 1$ is an integer and $X$ is an $S$-scheme, $X^{n}$ will denote the $S$-scheme defined recursively by $X^{1}=X$ and $X^{n}=X\times_{\be S} X^{n-1}$, where $n\geq 2$. By \cite[\S10.2, p.~97]{may}, $\{X^{n}\}_{\e n\e\geq\e 1}$ is a simplicial $S$-scheme with face maps
\begin{equation}\label{fa1}
\partial^{\,i}_{n}\,\colon\, X^{n+1}\to X^{n}, (x_{1},\dots,x_{n+1})\mapsto (x_{1},\dots,x_{i-1},x_{i+1},\dots, x_{n+1}),
\end{equation}
where $1\leq i\leq n+1$, i.e., $\partial^{\,i}_{n}$ is the canonical projection that omits the $i$-th factor. When $n=1$, we will use the following (standard) alternative notation for the maps \eqref{fa1}:
\begin{equation}\label{alt}
\partial^{\le 1}_{1}=p_{\le 2}\qquad\text{and}\qquad \partial^{\lle 2}_{1}=p_{\le 1},
\end{equation}
where $p_{\le r}\colon X^{2}\to X$ ($r=1,2$) is the projection onto the $r$-th factor.

Now assume that $S$ is locally noetherian, let $\s F$ be an abelian presheaf on $S_{\fl}$ and let $X\to S$ be a finite and faithfully flat morphism of schemes. Then we may consider the \v{C}ech cohomology groups of $\s F$ relative to the fppf covering $X\to S$, i.e., the cohomology groups $\check{H}^{\le i}(X\be/\be S, \s F\lle)$ of the complex of abelian groups $\big\{\s F\be\big(\be X^{n}\big), \partial^{\e n}\big\}_{\e n\e\geq\e 1}$, where
\[
\partial^{\, n}=\sum_{\e j=1}^{\e n+1}\e (-1)^{\le j+1}\lbe \s F\be\big(\partial_{ n}^{\, j}\big)\colon \s F\be\big(\be X^{ n}\big)\to \s F\be\big(\be X^{n+1}\big).
\]

If $\s F=\bg_{m,\e S}$ (respectively, $\s F=\mu_{\le n,\e S}$), the groups $\check{H}^{\le i}(X\be/\be S, \s F\lle)$ will be denoted by $\check{H}^{\le i}(X\be/\be S, \bg_{m})$ (respectively, $\check{H}^{\le i}(X\be/\be S, \mu_{\le n})$).

In terms of the standard notation \eqref{alt},
\begin{equation}\label{z}
\check{H}^{\le 0}(X\be/\lbe S, \s F\lle)=\krn[\s F\be\big(\le p_{\le 2}\lbe\big)\!\be-\be\!\s F\be\big(\le p_{\le 1}\lbe\big)\colon \s F(X)\to \s F(X^{2}\le)].
\end{equation}
Note that, since $f\lbe\circ\lbe p_{\le 2}=f\lbe\circ\lbe p_{\le 1}$, the map $\s F\lbe(\lle f\lle)\colon \s F(S\e)\to \s F(X)$ factors through $\check{H}^{\le 0}(X\be/\lbe S, \s F\lle)$, i.e.,
\begin{equation}\label{fact}
\s F\lbe(\lle f\lle)\colon \s F(S\e)\overset{\!\be\s F\lbe(\lle f\lle)^{\prime}}{\lra}\check{H}^{\le 0}(X\be/\lbe S, \s F\lle)\hookrightarrow \s F(X),
\end{equation}
where the second map is the inclusion.

The category of abelian presheaves on $S_{\fl}$ is abelian. Thus, given any integer $n\geq 1$, we may consider the presheaves $\s F_{\! n}$ and $\s F\be/\lbe n$ with values in the category of $n$-torsion abelian groups. Explicitly, $\s F_{\! n}$ and $\s F\be/\lbe n$ are defined as follows. If $T\to S$ is an object of $S_{\fl}$ and $h\colon T\to T^{\e\prime}$ is a morphism in $S_{\fl}$, then
\begin{eqnarray}
\s F_{\! n}(\lle T\lle)&=&\s F(\lle T\lle)_{n}\hskip .25cm\quad\text{and}\hskip .65cm\quad \s F_{\! n}(h)=\s F(h)_{n},\label{un}\\
(\s F\be/\lbe n)(\lle T\lle)&=&\s F(\lle T\lle)/n\quad\text{and}\quad (\s F\be/\lbe n)(h)=\s F(h)/n.\label{dois}
\end{eqnarray}
It follows from \eqref{kn}, \eqref{z} and \eqref{fact} that the following diagram exists and is commutative
\begin{equation}\label{mid}
\xymatrix{\s F_{\! n}(S\e)\ar@{=}[d]\ar[r]^(.4){\!\be\s F_{\lbe n}\lbe(\lle f\lle)^{\prime}}&\ar@{=}[d]\check{H}^{\le 0}(X\be/\lbe S, \s F_{\be n}\lle)\ar@{^{(}->}[r]& \s F_{\be n}(X)\ar@{=}[d]\\
\s F(S\e)_{n}\ar[r]^(.4){\!\be\s F\lbe(\lle f\lle)_{n}^{\prime}}&\check{H}^{\le 0}(X\be/\lbe S, \s F\lle)_{n}\ar@{^{(}->}[r]& \s F(X)_{n}.
}
\end{equation}
Thus
\begin{equation}\label{zn}
\s F_{\!n}\lbe(\lle f\lle)^{\prime}=\s F\lbe(\lle f\lle)^{\prime}_{n}.
\end{equation}
Note that, by the comment made right after \eqref{kn}, no analogous statements can be expected to hold for the presheaf $\s F\be/\lbe n$.

\smallskip

Next recall that, if $f\colon X\to S$ is endowed with a (right) action of a finite group $\Delta$, then $f$ is called a {\it Galois covering with Galois group $\Delta$} if the canonical map
\[
\coprod_{\e \delta\e\in\e \Delta}\be X\to X\!\times_{S}\!X, (x,\delta\e)\mapsto (x,x\delta\e),
\]
is an isomorphism of $S$-schemes. If $f$ is a Galois covering with Galois group $\Delta$ and $\s F$ transforms finite sums of schemes into direct products of abelian groups (e.g., is representable \cite[p.~231, line 7]{ega1}), then there exists a canonical isomorphism of abelian groups
\begin{equation}\label{gpc}
\check{H}^{\le i}(X\be/\be S, \s F\lle)\simeq H^{\le i}(\Delta, \s F(X)\lle)
\end{equation}
for every $i\geq 0$, where $\s F(X)$ is a left $\Delta$-module via the given right action of $\Delta$ on $X$ over $S$. See \cite[III, Example 2.6, p.~99]{mi1}.

\smallskip

\subsection{Restriction and corestriction maps} Let $f\colon S^{\e\prime}\to S$ be a morphism of schemes and let $X^{\prime}$ be an $\spp$-scheme. The {\it Weil restriction of $X^{\prime}$ along $f$} is the contravariant functor
$(\mathrm{Sch}/S)\to(\mathrm{Sets}), T\mapsto\hom_{
S^{\le\prime}}(T\times_{S}S^{\le\prime},X^{\le\prime}\le)$. The latter functor is representable if there exist an $S$-scheme $R_{S^{\le\prime}\be/S}(X^{\prime}\e)$ and a morphism of $S^{\e\prime}$-schemes $\theta_{X^{\prime},\e S^{\le\prime}\be/S}\colon R_{S^{\le\prime}\be/S}(X^{\prime}\e)_{S^{\le\prime}}\to X^{\prime}$ such that the map
\begin{equation}\label{wr}
\hom_{\le S}\e(T,R_{S^{\le\prime}\be/S}(X^{\le\prime}\e))\to\hom_{
S^{\le\prime}}(\e T\!\times_{S}\!S^{\e\prime},X^{\le\prime}\e), \quad g\mapsto \theta_{ X^{\prime}\!,\, S^{\le\prime}\be/S}\circ g_{\le S^{\le\prime}},
\end{equation}
is a bijection (functorially in $T\e$). See \cite[\S7.6]{blr} and \cite[Appendix A.5]{cgp} for basic information on the Weil restriction functor\,\footnote{\e As noted by Brian Conrad, the restriction in \cite[Appendix A.5]{cgp} to an affine base $S$ can be removed since all assertions in [loc.cit.] are local on $S$. Further, the noetherian hypotheses in [loc.cit.] are satisfied if $S$ and $\spp$ are locally noetherian.}\,. We will write
\begin{equation}\label{imor}
j_{ X,\e S^{\lle\prime}\be/ S}\colon X\to R_{\e S^{\lle\prime}\be/ S}(X_{\lbe S^{\lle\prime}}\lbe)
\end{equation}
for the canonical adjunction $S$-morphism, i.e., the $S$-morphism that corresponds to the identity morphism of $X_{S^{\lle\prime}}$ under the bijection \eqref{wr}.

Now assume that $S$ is locally noetherian and $f\colon S^{\e\prime}\to S$ is finite, faithfully flat and of constant rank $n\geq 1$ (in particular, $f$ is an fppf covering of $S\e$). Then $\spp$ is also locally noetherian and $j_{ S,\e S^{\lle\prime}\be/ S}$ is an isomorphism that will be regarded as an identification, i.e., we will write
\begin{equation}\label{cid}
R_{\e S^{\lle\prime}\be/ S}(\spp\e)=S.
\end{equation}
See \cite[Proposition A.5.7, p.~510]{cgp}. Now let $G$ be a commutative and quasi-projective $S$-group scheme. Then $j_{\le G,\e S^{\lle\prime}\be/ S}\colon G\to R_{\e S^{\lle\prime}\be/ S}(G_{\lbe S^{\lle\prime}}\lbe)$ \eqref{imor} is a closed immersion of commutative and quasi-projective $S$-group schemes. Further, if $G$ is smooth over $S$, then $R_{\e S^{\lle\prime}\be/ S}(G_{\lbe S^{\lle\prime}}\be)$ is smooth over $S$. See \cite[\S7.6, Theorem 4, p.~194]{blr}, \cite[II, Corollary 4.5.4]{ega} and \cite[Propositions A.5.2(4), A.5.7 and A.5.8, pp. p.~506-513]{cgp}.

\smallskip

For $r\geq 0$ and $\tau=\et$ or $\fl$, let $e_{\tau}^{\le (r)}\colon H^{\le r}\be(S_{\tau},R_{\e S^{\lle\prime}\be/ S}(G_{\lbe S^{\lle\prime}}\be)\e)\to H^{\le r}\be(\spp_{\tau},G\e)$ be the $r$-th edge morphism\,\footnote{ This is an instance of the first edge morphism in \cite[Proposition 2.3.1, p.~14]{t}.} induced by the Cartan-Leray spectral sequence associated to the pair $(G, f\le)$ relative to the $\tau$ topology:
\begin{equation}\label{leray}
H^{\le r}\be(S_{\tau},R^{\e s}\be f_{\lbe *}\lbe(G_{\be S^{\lle\prime}}\be))\implies H^{\e r+s}\lbe(\spp_{\tau},G\e).
\end{equation}
Note that, by \cite[Theorem 6.4.2(ii), p.~128]{t}, the maps $e_{\et}^{\lle r}$ are isomorphisms for every $r\geq 0$. However, the maps $e_{\fl}^{\lle r}$ are not isomorphisms in general \cite[XXIV, Remarks 8.5]{sga3}. We also note that, via the bijection \eqref{wr}, $e_{\tau}^{(0)}$ is the identity map: 
\begin{equation}\label{e0}
H^{\e 0}\lbe(S_{\tau},R_{\e S^{\lle\prime}\be/ S}(G_{\lbe S^{\lle\prime}}\be)\e)=R_{\e S^{\lle\prime}\be/ S}(G_{\lbe S^{\lle\prime}}\be)(S\e)=G_{\lbe S^{\lle\prime}}(\spp\e)=G(\spp\e)=H^{\e 0}\lbe(\spp_{\tau},G\e).
\end{equation}
The $r$-th restriction morphism associated to $(G,f\e)$ \eqref{ires}
factors as
\begin{equation}\label{jc}
j^{\le(r)}\colon H^{\e r}\lbe(S_{\tau},G\e)\overset{\!H^{r}\be(S_{\tau}\lbe,\e j\le)}{\lra} H^{\e r}\lbe(S_{\tau},R_{\e S^{\lle\prime}\be/ S}(G_{\lbe S^{\lle\prime}}\be)\e)\overset{\!e_{\tau}^{(r)}}{\lra} H^{\e r}\lbe(\spp_{\tau},G\e),
\end{equation}
where $j=j_{\e G,\e S^{\lle\prime}\be/ S}$.

\smallskip

Now, for every object $T\to S$ of $S_{\fl}$, set
\begin{equation}\label{abp}
\s H^{\e r}\be(G\e)(\le T\le)=H^{\le r}\lbe(\le T_{\fl},G\e).
\end{equation}
Further, if $h\colon T^{\e *}\to T$ is a morphism in $S_{\fl}$, set
\begin{equation}\label{abpm}
\s H^{\e r}\be(G\e)(h)=j_{\e G_{\le T},\e T^{\le *}\!/\e T,\e\fl}^{(r)}\colon H^{\le r}\lbe(\e T_{\fl},G\e)\to H^{\le r}\lbe(\le T^{\e *}_{\fl},G\e).
\end{equation}
Then \eqref{abp} and \eqref{abpm} define an abelian presheaf $\s H^{\e r}\be(G\e)$ on $S_{\fl}$ such that
\begin{equation}\label{cero}
\s H^{\le 0}\be(G\e)=G.
\end{equation}
By \eqref{fact}, $j_{\e G,\e\fl}^{(r)}=\s H^{\e r}\be(G\e)(\lle f\lle)$ factors as 
\begin{equation}\label{fr}
H^{\e r}\lbe(S_{\fl},G\e)\overset{ f_{G}^{(r)}}{\lra}\check{H}^{\le 0}\lbe(\spp\be/S,\s H^{\le r}\be(G\e)\lle)\hookrightarrow H^{\e r}\lbe(\spp_{\fl},G\e),
\end{equation}
where
\begin{equation}\label{frg}
f_{G}^{(r)}=\s H^{\e r}\be(G\e)(\lle f\lle)^{\prime}
\end{equation}
is the map in \eqref{fact} associated to the presheaf $\s F=\s H^{\e r}\be(G\e)$. Thus
\begin{equation}\label{kj}
\krn\e j_{\e G,\e\fl}^{(r)}=\krn[\, f_{G}^{(r)}\colon H^{\le r}\be(S_{\fl},G\e)\to \check{H}^{\le 0}\lbe(\spp\be/S,\s H^{\e r}\be(G\e)\lle)].
\end{equation}
The maps $f_{G}^{(r)}$ \eqref{frg} are the second edge morphisms in \cite[Proposition 2.3.1, p.~14]{t} associated to the spectral sequence for \v{C}ech cohomology \cite[Theorem 3.4.4(i), p.~58]{t}:
\begin{equation}\label{chec}
\check{H}^{s}(\spp\!/\lbe S,\s H^{\e r}\be(G\e)\lle)\implies H^{\e r+s}(S_{\fl},G\e).
\end{equation}
By \cite[p.~309, line 8]{mi1} and \eqref{kj} for $r=2$, the preceding spectral sequence yields an exact sequence of abelian groups  
\begin{equation}\label{lchec}
\begin{array}{rcl}
0&\to& \check{H}^{\le 1}(\spp\!/\be S, G\e)\overset{e^{\lle 1}_{\lle G}}{\lra} H^{\e 1}\lbe(S_{\fl},G\e)\overset{f_{G}^{(1)}}{\lra} \check{H}^{\le 0}(\spp\be/\lbe S, \s H^{\le 1}\lbe(G\e)\lle)\\
&\overset{d_{2}^{\e 0,1}}{\lra}& \check{H}^{\le 2}(\spp\!/\be S, G\e)\to\krn j_{\le G,\le\fl}^{(2)}\to \check{H}^{1}(\spp\!/S,\s H^{\le 1}\be(G\e)\lle)\to \check{H}^{3}(\spp\!/S,G\e),
\end{array}
\end{equation}
where $e^{\le 1}_{\le G}$ is an instance of the first edge morphism in \cite[Proposition 2.3.1, p.~14]{t}
and $d_{2}^{\, 0,1}$ is the first transgression map \cite[p.~49]{sha} induced by the spectral sequence \eqref{chec}. In particular, the exactness of the above sequence together with \eqref{kj} for $r=1$ implies the following well-known fact\,\footnote{\e The standard proof of this fact can be found, for example, in \cite[III, proof of Proposition 4.6, p.~123]{mi1}.}\e:

\begin{proposition} \label{ch} There exists a canonical isomorphism of abelian groups
\[
\krn\e j_{\le G, \e \fl}^{\le(1)}\isoto\check{H}^{\le 1}(\spp\!/\lbe S, G\e).
\]
\end{proposition}

\smallskip

Now let 
\begin{equation}\label{tmor}
N_{G,\e S^{\lle\prime}\be/ S}\colon R_{\e S^{\lle\prime}\be/ S}(G_{\lbe S^{\lle\prime}}\be)\to G
\end{equation}
be the norm (or trace) morphism  defined in \cite[XVII, 6.3.13.1 \& 6.3.14(a)]{sga4}. By \cite[XVII, Proposition 6.3.15(iv)]{sga4}, the composition
\begin{equation}\label{nis}
N_{G,\e S^{\lle\prime}\be/ S}\circ j_{G,\e S^{\lle\prime}\be/ S}=n_{\le G}\colon G\e\overset{j_{G,\e S^{\lle\prime}\be/ S}}{\lra} R_{\e S^{\lle\prime}\be/ S}(G_{\lbe S^{\lle\prime}}\be)\overset{\!N_{G,\e S^{\lle\prime}\be/ S}}{\lra} G
\end{equation}
is the $n$-th power morphism on $G$. The {\it $r$-th corestriction map associated to $G$ and $f\colon \spp\to S$}, denoted by 
\begin{equation}\label{nrm}
N^{(\e r\le)}\colon H^{\e r}\lbe(\spp_{\et},G\e)\to H^{\e r}\lbe(S_{\et},G\e),
\end{equation}
is the composition
\[
H^{\e r}\lbe(\spp_{\et},G\e)\underset{\!\sim}{\overset{\!(e_{\et}^{\le r})^{-1}}{\lra}}H^{\e r}\lbe(S_{\et},R_{\e S^{\lle\prime}\be/ S}(G_{\lbe S^{\lle\prime}}\be)\e)\overset{\!H^{\lle r}\be(S_{\et}\lbe,\e N\lle)}{\lra}H^{\e r}\lbe(S_{\et},G\e),
\]
where $N=N_{G,\e S^{\lle\prime}\be/ S}$.

Now, by the factorization \eqref{jc}, $N^{\le(\le r\le)}\circ j_{\le G,\e\et}^{\le(r)}=H^{\le r}\lbe(S_{\et}\lbe,\e N\le)\circ H^{\le r}\lbe(S_{\et}\lbe,\e j\le)$. Thus, by \eqref{nis}, the composition
\begin{equation}\label{nce}
H^{\le r}\lbe(S_{\et},G\e)\overset{\!j_{ G,\e\et}^{\le(r)}}{\lra} H^{\le r}\lbe(\spp_{\et},G\e)\overset{\!N^{\le(\le r\le)}}{\lra}  H^{\le r}\lbe(S_{\et},G\e)
\end{equation}
is the multiplication by $n$ map on $H^{\le r}\lbe(S_{\et},G\e)$. 
Consequently, $\krn\e j_{\le G,\e\et}^{\le(r)}\subseteq H^{\le r}\lbe(S_{\et},G\e)_{n}$, whence \eqref{sm} and \eqref{kj} yield the following statement.

\begin{lemma} \label{uu} If $G$ is smooth over $S$, then
\[
\krn\e j_{\le G}^{\le(r)}=\krn\e[\e H^{\le r}\lbe(S_{\et},G\e)_{n}\overset{\!\be \big(\le f_{\lbe G}^{\le(r)}\lbe\big)_{\!\be n}}{\lra} \check{H}^{\le 0}\lbe(\spp\be/S,\s H^{\le r}\be(G\e)\lle)_{n}\e],
\]
where $f_{G}^{(r)}=\s H^{\e r}\be(G\e)(\lle f\lle)^{\prime}$ is the map in \eqref{fr}.
\end{lemma}

\smallskip

\section{The capitulation kernel}\label{scap}

In this section we establish the main theorem of the paper (Theorem \ref{big}) which applies to the following types of pairs $(G,f\e)$:

\begin{definition}\label{adm} The pair $(G,f\e)$ is called {\it admissible} if \begin{enumerate}
\item[(i)] $S$ is locally noetherian,
\item[(ii)] $f\colon \spp\to S$ is finite, faithfully flat and of constant rank $n\geq 2$,
\item[(iii)] $G$ is smooth, commutative, quasi-projective, of finite presentation over $S$ and its fibers are connected, and
\item[(iv)] for every point $s\in S$ such that ${\rm char}\, k(\lbe s\lbe)$ divides $n$, $G_{\be\lle s}$ is a semiabelian variety over $k(\lbe s\lbe)$.
\end{enumerate}
\end{definition}

\begin{examples}\label{exs} The following are examples of admissible pairs:
\begin{enumerate}
\item[(a)] $f\colon \spp\to S$ is a finite and faithfully flat morphism of connected noetherian schemes and $G$ is a semiabelian $S$-scheme. For example,  $G=\bg_{m,\e S}$.
\item[(b)]  $f\colon \spp\to S$ is a finite and faithfully flat morphism of rank $n\geq 2$ between connected Dedekind schemes, $F$ is the function field of $S$, $G_{\lbe F}$ is a smooth, commutative and connected $F$-group scheme of finite type,  $G_{\lbe F}$ has a N\'eron $S$-model $\s N$ with semiabelian reduction at all $s\in S$ such that ${\rm char}\, k(\be s\be)$ divides $n$ and $G=\s N^{\e 0}$ is the identity component of $\s N$. That $G$ satisfies condition (iii) of Definition \ref{adm} follows from \cite[\S6.4, Theorem 1, p.~153, and \S10.1, p.~290, line 6]{blr}.
\end{enumerate}
\end{examples}

\smallskip

\begin{proposition}\label{isg} Let $n\geq 1$ be an integer, let $S$ be a locally noetherian scheme and let $G$ be a smooth and commutative $S$-group scheme with connected fibers. Assume that, for every point $s\in S$ such that ${\rm char}\, k(\lbe s\lbe)$ divides $n$, $G_{\lbe s}$ is a semiabelian $k(\lbe s\lbe)$-variety. Then $n_{\le G}\colon G\to G$ is faithfully flat and locally of finite presentation.
\end{proposition}
\begin{proof} By \cite[(6.2.1.2) and Proposition 6.2.3(v), p.~298]{ega1}, $n_{\le G}$ is locally of finite presentation. Now, by \cite[${\rm IV_{3}}$, Corollary 11.3.11]{ega}, to establish the flatness and surjectivity of $n_{\le G}$ we may assume that $S=\spec F$, where $F$ is a field. If $n$ is prime to ${\rm char}\, F$, then $n_{\le G}$ is \'etale and therefore flat \cite[${\rm VII_{A}}$, \S8.4, Proposition]{sga3}. Thus, since $G$ is connected, $n_{\le  G}$ is surjective by \cite[${\rm VI_{B}}$, Proposition 3.11 and its proof]{sga3}. Assume now that ${\rm char}\, F$ divides $n$ and let $\bar{F}$ be an algebraic closure of $F$. Since the $n$-th power morphism $\bar{F}^{\e *}\to \bar{F}^{\e *}$ is surjective, $n_{\le G}\lbe(\bar{F}\lle)$ is surjective if $G$ is an $F$-torus. If $G$ is an abelian variety over $F$, then $n_{\le G}\lbe(\bar{F}\lle)$ is also surjective by \cite[p.~62]{mum}. Now let $G$ be a semiabelian variety over $F$, i.e., an extension $1\to T\to G\to A\to 1$, where $T$ is an $F$-torus and $A$ is an abelian variety over $F$. Then the following exact and commutative diagram of abelian groups shows that the map $n_{\lle G}\lbe(\bar{F}\lle)\colon G(\le\bar{F}\e)\to G(\le\bar{F}\e)$ is surjective as well:
\[
\xymatrix{
0\ar[r]& T(\le\bar{F}\e)\ar@{->>}[d]^(.45){n_{\lle T}\lbe(\bar{F}\lle)}\ar[r]& G(\le\bar{F}\e)\ar[r]\ar[d]^(.45){n_{\lle G}\lbe(\bar{F}\lle)}& A(\le\bar{F}\e)\ar@{->>}[d]^(.45){n_{\lbe A}\lbe(\bar{F}\lle)}\ar[r]& 0\\
0\ar[r]&T(\le\bar{F}\e)\ar[r]& G(\le\bar{F}\e)\ar[r]& A(\le\bar{F}\e)\ar[r]& 0.
}
\]
We conclude from \cite[I, \S3, Corollary 6.10, p.~96]{dg} that $n_{\le G}$ is surjective. Thus, since $G$ is smooth, $n_{\le G}$ is flat by \cite[${\rm VI_{B}}$, Proposition 3.11]{sga3}.
\end{proof}

Let $f\colon \spp\to S$ be a morphism of schemes and let $G$ be an $S$-group scheme such that the pair $(G,f\e)$ is admissible (see Definition \ref{adm}). Let $n\geq 2$ be the rank of $f$. By Proposition \ref{isg},
\begin{equation}\label{bas}
0\to G_{\lbe n}\to G\overset{\!n}{\to}G\to 0
\end{equation}
is an exact sequence in $\sfs$. For every object $T\to S$ of $S_{\fl}$ and every integer $r\geq 1$, \eqref{bas} induces an exact sequence of abelian groups
\begin{equation}\label{ab}
0\to H^{\le r-1}\lbe(T_{\et},G\e)/n\to H^{\le r}(T_{\fl}, G_{\lbe n}\le)\to  H^{\le r}\lbe(T_{\et},G\e)_{n}\to 0,
\end{equation}
where the left-hand nontrivial map above is induced by the connecting morphism $H^{\le r-1}\lbe(T_{\fl},G\e)\simeq H^{\le r-1}\lbe(T_{\et},G\e)\to H^{\le r}(T_{\fl}, G_{\lbe n}\le)$ in fppf cohomology induced by the sequence \eqref{bas}.
Thus there exists a canonical exact sequence of abelian presheaves on $\sfs$
\begin{equation}\label{abs}
0\to \s H^{\e r-1}\lbe(G\e)/n\to \s H^{\e r}\be(G_{n})\to \s H^{\e r}\be(G\e)_{n}\to 0,
\end{equation}
where the left-hand term is the presheaf \eqref{dois} associated to  
$\s H^{\e r-1}\lbe(G\e)$, the middle term is the presheaf \eqref{abp} associated to $G_{\lbe n}$ and the right-hand term is the presheaf \eqref{un} associated to $\s H^{\e r}\lbe(G\e)$. By \eqref{fact}, \eqref{zn} and \eqref{frg}, the maps $(\s H^{\le r-1}\lbe(G\e)/n)(\lle f\le)$ and $\s H^{\e r}\lbe(G\e)_{n}(\lle f\le)$ factor as
\begin{equation}\label{map1}
H^{\le r-1}\lbe(\le S_{\et},G\e)/n\overset{\le f^{(r-1)}_{/n}}{\lra}\check{H}^{\le 0}(\spp\!/\lbe S,\s H^{\e r-1}\be(G\e)/n\lle)\hookrightarrow H^{\e r-1}\lbe(\le \spp_{\et},G\e)/n
\end{equation}
and 
\begin{equation}\label{map2}
H^{\le r}\lbe(\le S_{\et},G\e)_{n}\overset{\!\be\big(f^{(r)}_{\lbe G}\be\big)_{\!\be n}}{\lra}\check{H}^{\le 0}(\spp\!/\lbe S,\s H^{\le r}\be(G\e))_{n}\hookrightarrow H^{\le r}\lbe(\le \spp_{\et},G\e)_{n},
\end{equation}
respectively, where $f^{(r-1)}_{/n}=(\s H^{\e r-1}\lbe(G\e)/n)(\lle f\lle)^{\prime}\,$\footnote{\e Note that, in general, $f^{(r-1)}_{/n}\neq f_{G}^{(r-1)}\be/n$, by the comment following \eqref{zn}.}\,.

\smallskip

Now set $\sppp=\spp\times_{S}\spp$ and, for $i=1$ or $2$, let $p_{\le i}\colon \sppp\to\spp$ be the projection onto the $i$-th factor. We will write $p_{\lle i}^{\lle *}=G(\e p_{\lle i})\colon G(\spp\le)\to G(\sppp\le)$ for the morphism of abelian groups induced by $p_{\le i}$. We now define
\begin{equation}\label{bgn}
\Psi(G,f\e)=\{\e x\in G(\spp\le)\colon p_{\lle 1}^{\lle *}(x)\le p_{\lle 2}^{\lle *}(x)^{-1}\be\in\lbe G(\sppp\le)^{n}\e\}.
\end{equation}

When $G=\bg_{m,\le S}$, \eqref{bgn} will be written 
\begin{equation}\label{ugn}
\Psi(U,f\e)=\{\e x\in U(\spp\le)\colon p_{\lle 1}^{\lle *}(x)\le p_{\lle 2}^{\lle *}(x)^{-1}\be\in\lbe U(\sppp\le)^{n}\e\},
\end{equation}
where $U$ is the global units functor \eqref{gu}.

Since $f\circ p_{\le 1}=f\circ p_{\le 2}$, the image of $G(\lle f\lle)\colon G(S\e)\hookrightarrow G(\spp\le)$ is contained in $\Psi(G,f\e)$. Thus \eqref{bgn} is a subgroup of $G(\spp\le)$ containing $G(S\le)\le G(\spp\le)^{n}$. By \eqref{z}, \eqref{dois} and \eqref{cero}, we have
$\check{H}^{\le 0}(\spp\be/\lbe S,\s H^{\e 0}\be(G\e)/n\lle)=\Psi(G,f\e)/G(\spp)^{n}$ and
$f^{(0)}_{/n}\colon G(S\le)/n\to \check{H}^{\le 0}(\spp\be/\lbe S,\s H^{\e 0}\be(G\e)/n\lle)$ is the map $G(S\le)/n\to \Psi(G,f\e)/G(\spp)^{n}$ induced by $G(\lle f\lle)\colon G(S\e)\hookrightarrow G(\spp)$. Thus there exist canonical isomorphisms of abelian groups
\begin{equation}\label{pri}
\krn f^{(0)}_{/n}\simeq\frac{G(S\e)\cap G(\spp)^{n}}{G(S\e)^{n}}
\end{equation}
and
\begin{equation}\label{seg}
\cok f^{(0)}_{/n}\simeq \frac{\Psi(G,f\e)}{G(S\e)G(\spp)^{n}},
\end{equation}
where the intersection takes place inside $G(\spp)$.

We now consider the exact and commutative diagram of abelian groups
\begin{equation}\label{diag}
\xymatrix{0\ar[r]& H^{\le r-1}\lbe(S_{\et},G\e)/n\ar[r]\ar[d]_{f^{(r-1)}_{/n}} & H^{\le r}(S_{\fl}, G_{\lbe n}\le)\ar[d]_{f_{G_{n}}^{(r)}}\ar@{->>}[r]& H^{\le r}(S_{\et}, G\e)_{n}\ar[d]\ar[d]_{\big(f^{(r)}_{\lbe G}\be\big)_{\! n}}\\
0\ar[r]&\check{H}^{\le 0}(\spp\!/\lbe S,\s H^{\le r-1}\be(G\e)/n\lle)\ar[r]& \check{H}^{\le 0}(\spp\!/\lbe S,\s H^{\e r}\be(G_{n})\lle)\ar[r]&\check{H}^{\le 0}(\spp\!/\lbe S,\s H^{\e r}(G\e)\lle)_{n}&&
}
\end{equation}
where the top row is the sequence \eqref{ab} for $T=S$, the bottom row is the beginning of the \v{C}ech cohomology sequence induced by \eqref{abs} \cite[p.~97, line $-4$]{mi1} and the vertical maps are given, respectively, by \eqref{map1}, \eqref{fr} (for $G_{n}$) and \eqref{map2} via the equality $\check{H}^{\le 0}(\spp\!/\lbe S,\s H^{\e r}(G\e)_{n}\lle)=\check{H}^{\le 0}(\spp\!/\lbe S,\s H^{\e r}(G\e)\lle)_{n}$ \eqref{mid}. Using \eqref{kj} for $G_{n}$ and Lemma \ref{uu}\e, the above diagram yields an exact sequence of abelian groups
\begin{equation}\label{si1}
0\to \krn f^{(r-1)}_{/n}\to\krn j_{G_{n},\le\fl}^{\le(r)}\to \krn\e j_{\le G}^{\le(r)}\to \cok f^{(r-1)}_{/n}\to \cok f_{G_{n}}^{(r)}.
\end{equation}
We now observe that \eqref{lchec} (applied to $G_{n}$) induces an injection $\cok f_{G_{ n}}^{(1)}\hookrightarrow \check{H}^{\le 2}(\spp\!/\be S, G_{\lbe n}\lbe)$. Thus  Proposition \ref{ch}, \eqref{pri}, \eqref{seg} and the exactness of \eqref{si1} for $r=1$ yield the main result of this paper:

\begin{theorem}\label{big} Assume that the pair $(G,f\le)$ is admissible (see Definition {\rm \ref{adm}}) and let $n\geq 2$ be the rank of $f$. Then there exists a canonical exact sequence of $n$-torsion abelian groups
\[
\begin{array}{rcl}
0&\to& G(S\e)\be\cap\be G(\spp\le)^{n}\be/G(S\e)^{n}\to \check{H}^{\le 1}(\spp\!/\be S, G_{\lbe n}\lbe)\to \krn\e j_{\le G}^{\le(1)}\\
&\to& \Psi(G,f\e)/G(S\e)\le G(\spp\le)^{n}\to \check{H}^{\le 2}(\spp\!/\be S, G_{\lbe n}\lbe),
\end{array}
\]
where $j_{\le G}^{\le(1)}\colon H^{\le 1}(S_{\et}, G\e)\to H^{\le 1}(\spp_{\et}, G\e)$ is the first restriction map associated to $(G,f\le)$ \eqref{cap1} and $\Psi(G,f\le)$ is the subgroup \eqref{bgn} of $G(\spp\le)$. If, in addition, $f$ is a Galois covering with Galois group $\Delta$ and the groups $\check{H}^{\le i}(\spp\!/\be S, G_{\lbe n}\lbe)$ above are replaced with $H^{\le i}\lbe(\Delta, G(\spp\le)_{n}\lbe)$ via \eqref{gpc}, where $i=1$ and $2$, then the resulting sequence is also exact.
\end{theorem}

\begin{remark}\label{maps} The nontrivial maps in the sequence of the theorem can be described as follows. The map $G(S\e)\le\cap\le G(\spp\le)^{n}\be/G(S\e)^{n}\to \check{H}^{\le 1}(\spp\!/\be S, G_{\lbe n}\lbe)$ is the composition of the inverse of the isomorphism \eqref{pri} and the top horizontal arrow in the following commutative square
\[
\xymatrix{\krn f^{(0)}_{/n}\ar@{^{(}->}[d]\,\ar@{^{(}->}[r]&  \check{H}^{\le 1}(\spp\!/\be S, G_{\lbe n}\lbe)\ar@{^{(}->}[d]^(.45){e^{\lle 1}_{\lle G_{\lbe n}}}\\
G(S\e)/n \,\ar@{^{(}->}[r]&  H^{\le 1}(S_{\fl},G_{\lbe n}),
}
\] 
where $e^{\le 1}_{\le G_{\lbe n}}$ is the edge morphism in the exact sequence \eqref{lchec} associated to $G_{n}$ and the bottom horizontal arrow is the left-hand nontrivial map in \eqref{ab} for $r=1$ and $T=S$. The map $\check{H}^{\le 1}(\spp\!/\be S, G_{\lbe n}\lbe)\to \krn\e j_{\le G}^{\le(1)}$ is the composition of the map $\check{H}^{\le 1}(\spp\!/\be S, G_{\lbe n}\lbe)\to \check{H}^{\le 1}(\spp\!/\be S, G\e)$ induced by the immersion $G_{n}\hookrightarrow G$ and the isomorphism $\check{H}^{\le 1}(\spp\!/\be S, G\e)\isoto \krn\e f_{\le G}^{\le(1)}=\krn\e j_{\le G}^{\le(1)}$ \eqref{kj} induced by the edge morphism $e^{\le 1}_{\le G}$ in \eqref{lchec} (note that the latter map is the inverse of the isomorphism of Proposition \ref{ch}). The map $\krn\e j_{\le G}^{\le(1)}\to\Psi(G,f\le)/G(S\e)\le G(\spp)^{n}$ is the composition of the connecting morphism $\krn\e j_{\le G}^{\le(1)}=\krn\e \big(\le f_{\lbe G}^{\le(1)}\lbe\big)_{\! n}\to \cok f^{(0)}_{/n}$ in the snake-lemma exact sequence arising from diagram \eqref{diag} and the isomorphism \eqref{seg}. Finally, $\Psi(G,f\le)/G(S\e)\le G(\spp)^{n}\to \check{H}^{\le 2}(\spp\!/\be S, G_{\lbe n}\lbe)$ is the composition of the inverse of the isomorphism \eqref{seg}, the map $\cok f^{(0)}_{/n}\to \cok f_{G_{ n}}^{(1)}$ induced by  $\check{H}^{\le 0}(\spp\!/\lbe S,\varphi)$, where $\varphi\colon \s H^{\le 0}\be(G\e)/n\to \s H^{\le 1}\be(G_{n})$ is the left-hand nontrivial map in the exact sequence \eqref{abs} for $r=1$, and the injection $\cok f_{G_{ n}}^{(1)}\hookrightarrow \check{H}^{\le 2}(\spp\!/\be S, G_{\lbe n}\lbe)$ induced by the transgression map $d_{2}^{\,0,1}$ appearing in the exact sequence \eqref{lchec}.
\end{remark}

\smallskip

Via the identifications $H^{\e 1}\lbe(S_{\et},\bg_{m})=\pic S$ and $H^{\e 1}\lbe(\spp_{\et},\bg_{m})=\pic \spp$, the first restriction map
$j_{\e\bg_{m,\e S}}^{(1)}\colon H^{\le 1}(S_{\et}, \bg_{m})\to H^{\le 1}(\spp_{\et}, \bg_{m})$ \eqref{cap1} associated to the pair $(\bg_{m,\e S},f\e)$ can be identified with the canonical map $\pic f\colon \pic S\to \pic \spp$ \eqref{pmap}. Further, by Example \ref{exs}(a), $(\bg_{m,\e S},f\e)$ is an admissible pair. Thus the theorem immediately yields

\begin{corollary}\label{mcor} Let $f\colon \spp\to S$ be a finite and faithfully flat morphism of rank $n\geq 2$ between connected noetherian schemes.
Then there exists a canonical exact sequence of $n$-torsion abelian groups
\[
\begin{array}{rcl}
1&\to& U(S\e)\be\cap\be U(\spp\le)^{n}\be/U(S\e)^{n}\to \check{H}^{\le 1}(\spp\!/\be S, \mu_{\e n}\lbe)\to \krn\e \pic\be f\\
&\to& \Psi(U,f\le)/U(S\e)\e U(\spp\le)^{n}\to \check{H}^{\le 2}(\spp\!/\be S,  \mu_{\e n}\lbe),
\end{array}
\]
where $U$ is the global units functor \eqref{gu}, $\pic\be f$ is the canonical map \eqref{pmap} and $\Psi(U,f\le)\subset U(\spp\le)$ is the group \eqref{ugn}.
\end{corollary}

\section{Capitulation of ideal classes}\label{last}

In this final section we specialize Corollary \ref{mcor} to the case $S=\spec \ofs$. The notation is as in the Introduction.

\smallskip

As noted in the Introduction, the capitulation map \eqref{ccap} can be identified with the canonical map $\pic\be f$ \eqref{pmap}.

Now observe that $\sppp=\spp\times_{S}\spp=\spec(\oks\le\otimes_{\le\ofs}\le\oks\le)$ and the canonical projection morphisms
$p_{\le i}\colon \sppp\to\spp$ for $i=1$ and $2$ are induced, respectively, by the morphisms of rings $\oks\to\okfs$ given by $x\mapsto x\otimes 1$ and $x\mapsto 1\otimes x$, where $x\in\oks$. Thus, if $G=\bg_{m,\e S}$, then the morphisms of abelian groups $p_{\lle i}^{\lle *}=G(\e p_{\lle i})\colon G(\spp\le)\to G(\sppp\le)$ induced by $p_{\le i}$ are the maps $\oks^{\e *}\to (\okfs)^{*}$ defined by $u\mapsto u\otimes 1$ and $u\mapsto 1\otimes u$, where $u\in\oks^{\e *}$. Consequently, the group $\Psi(U,f\e)$ \eqref{ugn} can be identified with
\begin{equation}\label{cool}
\Psi(K/F,\Sigma\e)=\{\e u\in \oks^{\e *}\colon u\otimes u^{-1}\in ((\okfs)^{*})^{n}\e\}.
\end{equation}

Thus Corollary \ref{mcor} yields the following statement.

\begin{proposition}\label{mcor2} There exists a canonical exact sequence of finite $n$-torsion abelian groups
\[
\begin{array}{rcl}
0&\to& \ofs^{\e *}\be\cap\be (\oks^{\e *}\e)^{n}\be/(\ofs^{\e *}\e)^{n}\to \check{H}^{\le 1}(\oks\e/\ofs, \mu_{\le n})\to \krn\e j_{K\be/\lbe F,\e\si}\\
&\to& \Psi(K/F,\Sigma\le)/\ofs^{\e *}\e (\oks^{\e *})^{n}\to \check{H}^{\le 2}(\oks\e/\ofs, \mu_{\le n}),
\end{array}
\]
where $j_{K\be/\lbe F,\e\si}$ is the capitulation map \eqref{ccap} and $\Psi(K/F,\Sigma\e)$ is the subgroup of $\oks^{\e *}$ given by \eqref{cool}.
\end{proposition}

If $\Sigma$ contains all primes of $F$ that ramify in $K$, then $f\colon \spec \oks\to\spec \ofs$ is a Galois covering with Galois group $\Delta$. Thus the following statement is immediate from the proposition using \eqref{gpc}.

\begin{corollary}\label{mcor3} Let $K/F$ be a finite Galois extension of global fields with Galois group $\Delta$ of order $n$ and let $\Sigma$ be a nonempty finite set of primes of $F$ containing all archimedean primes and all primes that ramify in $K$. Then there exists a canonical exact sequence of finite $n$-torsion abelian groups
\[
\begin{array}{rcl}
0&\to& \ofs^{\e *}\be\cap\be (\oks^{\e *}\le)^{n}\be/(\ofs^{\e *}\le)^{n}\to H^{\le 1}(\Delta, \mu_{\le n}(K\e))\to \krn\e j_{K\be/\lbe F,\e\si}\\
&\to& \Psi(K/F,\Sigma\le)/\ofs^{\e *}\e (\oks^{\e *})^{n}\to H^{\le 2}(\Delta, \mu_{\le n}(K\e)).
\end{array}
\]
Consequently, if $n$ is odd and $\mu_{\le n}(K\e)=\{1\}$, then there exists  a canonical isomorphism of finite $n$-torsion abelian groups
\[
\krn\e j_{K\be/\lbe F,\e\si}\isoto \Psi(K/F,\Sigma\le)/\ofs^{\e *}\e (\oks^{\e *})^{n}.
\]
In particular, at most $[\e\oks^{\e *}\e\colon \ofs^{\e *}\e (\oks^{\e *})^{n}\e]$ $\Sigma$-ideal classes of $F$ capitulate in $K$. 
\end{corollary}

\smallskip

\begin{remark} The group $\Psi(K/F,\Sigma\e)$ \eqref{cool} seems difficult to compute in general since the structure of $(\okfs)^{*}$ is rather mysterious. Perhaps a more sensible approach would be to investigate whether or not $\Psi(K/F,\Sigma\e)$ satisfies a Hasse principle, i.e., whether the condition that defines $\Psi(K/F,\Sigma\e)$, namely $u\otimes u^{-1}\in ((\okfs)^{*})^{n}$, is determined locally. In general, if $B$ is an $A$-algebra, where $A$ and $B$ are (commutative and unital) rings, then there exists a canonical map $B^{\e*}\times B^{\e*}\to (B\otimes_{\lbe A}\be B\e)^{*},(b_{\le 1},b_{\le 2})\mapsto b_{\le 1}\otimes b_{\le 2},$ which is not surjective in general. When $A$ is a field and $B$ and $C$ are finitely generated $A$-algebras, the structure of $(B\otimes_{\lbe A} C\e)^{*}$ has been discussed by Jaffe \cite{jaf}. However, no general structure theorem for $(B\otimes_{\lbe A}\be B\e)^{*}$ seems to be known. Of course, it may still be possible to determine the subgroup $\Psi(K/F,\Sigma\e)$ of $\oks^{\e *}$ explicitly is some particular cases of interest and then apply Corollary \ref{mcor3} to obtain nontrivial information on the capitulation kernel, e.g., when  $K/F$ is a (possibly ramified) quadratic extension of number fields. For the quadratic unramified case (when $F$ is a totally real quadratic number field), see \cite{bss}.  
\end{remark}

\end{document}